\documentclass[a4paper]{amsproc}
\usepackage[colorlinks]{hyperref}

\usepackage[T1]{fontenc}
\usepackage[utf8]{inputenc}
\usepackage[shortlabels]{enumitem}
\usepackage{todonotes}
\usepackage[mathscr]{eucal}
\usepackage{mathrsfs,dsfont}

\usepackage{tkz-fct}
\usepackage{tikz}
\usepackage{tikz-cd}
\usepackage{tikz-3dplot}
\usetikzlibrary{calc,hobby}
\usetikzlibrary{decorations.pathmorphing,patterns,decorations,shapes,arrows, intersections,matrix,fit,calc,trees,positioning,arrows,chains, shapes.geometric,shapes,angles,quotes,fit,math}
\usepackage{pgf,tikz}
\usetikzlibrary{through,backgrounds}

\newtheorem{theorem}{\rm\bf Theorem}[section]
\newtheorem{proposition}[theorem]{\rm\bf Proposition}
\newtheorem{lemma}[theorem]{\rm\bf Lemma}

\newtheorem*{theorem*}{Theorem}
\newtheorem*{theorem 1}{\rm\bf Proposition 1}
\newtheorem*{theorem 2}{\rm\bf Proposition 2}
\newtheorem*{conj*}{Conjecture}

\theoremstyle{definition}

\theoremstyle{remark}
\newtheorem{remark}[theorem]{\rm\bf Remark}

\begin{document}

\title{Density-dependent feedback in age-structured populations}

\author[J. Andersson, V. Kozlov, V. Tkachev et al]{J. Andersson}
\address{Department of Mathematics, Link\"oping University, Sweden}
\email{jonathan.andersson@liu.se}
\author[]{V. Kozlov}
\address{Department of Mathematics, Link\"oping University, Sweden} \email{vladimir.kozlov@liu.se}           
\author[]{V.G. Tkachev}
\address{Department of Mathematics, Link\"oping University, Sweden}
\email{vladimir.tkatjev@liu.se}           
\author[]{S. Radosavljevic}
\address{Stockholm Resilience Centre , Stockholm University, Sweden}
\email{sonja.radosavljevic@su.se}
\author[]{U. Wennergren}
\address{Department of Physics, Chemistry, and Biology, Link\"oping University}
\email{uno.wennergren@liu.se}

\begin{abstract}
The population size has far-reaching effects on the fitness of the population, that, in its turn  influences the population extinction or persistence. Understanding the density- and age-dependent factors  will facilitate more accurate predictions about the population dynamics and its asymptotic behaviour.  In this paper, we develop a rigourous mathematical analysis to study  positive and negative effects of increased population density in the classical nonlinear age-structured population model introduced by Gurtin \& MacCamy in the late 1970s. One of our main results expresses the global stability of the system in terms of the newborn function only. We also derive the existence of a threshold population size implying the population extinction, which is well-known in population dynamics as an Allee effect.

\end{abstract}

\maketitle





\section{Introduction}

In biological populations, density-dependent regulation represents change in individual fitness caused by changes in population size or density.
The negative density-dependency, often explained by intra-specific competition and overcrowding effect, is characterized by decline in fitness with increase in populations size or density. In sharp contrast with this is the positive density-dependency, or the Allee effect, characterized by increase in fitness with increase in population size. Various mechanisms have been considered as a source of the Allee effect, \cite{Berec,Courchamp,Courchamp99}, pointing out that increase in fitness can come though increase in birth rate, decrease in death rate or both.

Mathematical models of age structured populations usually use density dependent vital rates without any special regard to the type of feedback that density-dependence produces; see for example \cite{BrauerChavez,Cip,DGH,Eld,Grt,Webb}. On the other hand, some authors investigate consequences of the Allee effect in age-structured populations, see for instance \cite{Cus1,Cus2}, or intraspecific competition  \cite{we2}.

The importance of this article is twofold. First, we expand mathematical theory of age-structured population dynamics by including density-dependent regulation. Second, Allee effect may have a positive contribution to population survival. In the age of massive extinction of species, it is therefore important to study under which conditions population may survive.

In this paper we study consequences of different types of density-dependence on permanence of age-structured populations.  It is known that the general problem is descibed by the  Gurtin-MacCamy delay equations and the stability of steady states is governed by nonlinear characteristic equations that are not easy to analyze analytically; see, however, a recent book of Iannelli and Milner \cite{Iannelli} where stability of equilibria and local asymptotic behavior of solutions is studied a model with several significant variables (the so-called \textit{sizes}). Typical sizes are the total population, or number of juveniles or adults at time $t$.

Using the framework  proposed in  \cite{Iannelli}, we improve our previous assumptions used in \cite{we2}, \cite{KRTW} that intraspecific competition occurs only among individuals of the same age by using more realistic age, and the density-dependent mortality $\mu(a,P(t))$ and the fertility $\beta(a,Q(t))$, where
$$
P(t)=\int_0^{\infty}p(a)n(a,t)\,da,\qquad  Q(t)=\int_0^{\infty}q(a)n(a,t)\,da
$$ are weighted populations. Here $n(a,t)$ denotes the number of individuals of age $a$ at time $t$ and $p(a)$, $q(a)$ are certain weight functions. For example, the choice $p(a)\equiv 1$ for $A_1\le a\le A_2$ and $p(a)\equiv 0$ otherwise,  yields the population size at the time $t$ in the age interval $[A_1,A_2]$.

One of our main assumptions is that mortality rate tends to infinity with the population size. This assumption is having a biological explanation: intraspecific competition is increasing in any large population due to limited resources in the habitat. Important consequence of this assumption, stated in Section \ref{boundedness}, is existence of an upper bound for a population. Moreover, this result is an improvement of a similar result in \cite{bound}, which is made possible by allowing two different weight functions for the mortality and fertility rates and by relaxing condition of Lipschitz continuity for the weight functions.

In Section \ref{stability} a stability analysis is performed on the trivial equilibrium $(\rho,P,Q)=(0,0,0)$. The stability of the trivial equilibrium depends on the net reproductive rate $$R_0=\int_0^\infty \beta(a,0)e^{-\int_0^a\mu(v,0)dv}da.$$

Our main result expresses the global stability of the system in terms of the newborn function only. More specifically, in Section \ref{largetimebehavior} we investigate the global stability of the system in terms of newborns only. We restrain the mortality rate to be increasing with $P$ and thus we do not incorporate the Allee effect on the mortality. Under this assumption we derive conditions based on the net reproduction rate $R_0$ for extinction and persistence. In the case $R_0\leq 1$ the population will go extinct and in the case $R_0>1$ the population will be persistent.

In Section \ref{permanence} we remove the restriction on the mortality function made in chapter \ref{largetimebehavior}. This corresponds to such an intraspecific mechanism as the Allee effect. More precisely,  if $R_0<1$ we conclude that the population either becomes extinct or is persistent. We note that if the number of newborns ever is small enough then this implies extinction. This effectively means that the trivial equilibrium is locally stable.

\section{The model setup}
Density dependent regulation acts on a population by changing its birth and death rates. Gurtin and MacCamy \cite{Grt} and Chipot \cite{Cip} assumed that the strength of density dependent regulation always depend on the total population, while Kozlov et al. \cite{we2} took the opposite approach by assuming that competition occurs only within each age-class. Here, we will follow the model from Chapter 5 of \cite{Iannelli} with some restrictions. In order to encompass various mechanisms through which density dependent regulation can manifest, we introduce the weighted age-class functions
\begin{align}\label{p}
P(t)=\int_0^{\infty}p(a)n(a,t)\,da,
\end{align}
and
\begin{align}\label{q}
Q(t)=\int_0^{\infty}q(a)n(a,t)\,da,
\end{align}
where $n(a,t)$ is the number of individuals of age $a$ at time $t$ and $p(a)$ and $q(a)$ are non-negative weight functions. The balance equation is then:
\begin{align}\label{gen1}
   \frac{\partial{n(a,t)}}{\partial t}+\frac{\partial{n(a,t)}}{\partial a} &= -\mu(a,P(t))n(a,t), \quad a,t>0,
\end{align}
where the function $\mu(a,P(t))$ is the death rate dependent on the weighted age-class function $P(t)$. The boundary condition is given by
\begin{align}\label{genbc1}
n(0,t)=\int_0^{\infty}\beta(a,Q(t))n(a,t)\,da, \quad t> 0,
\end{align}
where the birth rate $\beta(a,Q(t))$ incorporates effect of age-class density through the weighted age-class function $Q(t)$. The initial condition is given by:
\begin{align}\label{genic1}
n(a,0)=g(a), \quad a>0.
\end{align}

The boundary-initial value problem (\ref{gen1})--(\ref{genic1}), together with the weighted age-class-functions (\ref{p}) and (\ref{q}), constitutes a density-dependent population growth model. For purposes of our analysis and in line with the theory in Chapter 5 of \cite{Iannelli}, we assume that the parameters satisfy following conditions:

\begin{enumerate}[label=\textrm{($H_{\arabic*}$)},series=lafter]
\item\label{H'1}
The function $\mu(a,x)$ is assumed to be of the form
\begin{equation}
\mu(a,x)=\mu_0(a)+\mathscr{M}(a,x),
\end{equation}
where for some $a_\dagger>0$
\begin{equation}
\mu_0\in L_{loc}^1([0,a_{\dagger}))\text{, }\quad\mu_0(a)\geq 0 \quad \text{a.e. in } [0,a_\dagger]\text{, }\int_0^{a_\dagger}\mu_0(\sigma)d\sigma=+\infty
\end{equation}
and $\mathscr{M}(\cdot,x)$ is a continuous operator that for each $x\in\mathbb{R}_+=
\{x\in\mathbb{R}:x\geq 0\}$, gives a function in $L^1(0,a_\dagger)$, that is
$$\mathscr{M}(\cdot,x)\in C(\mathbb{R}_+,L^1(0,a_\dagger)).$$
We also assume that
\begin{align}
\mathscr{M}(a,x)&\geq 0\quad \text{a.e. in } [0,a_\dagger]\times\mathbb{R}_+ \\
\intertext{and }
\mathscr{M}(a,0)&=0\quad \text{a.e. in }[0,a_\dagger].\nonumber
\end{align}

\item\label{H'2}
The function $\beta$ satisfies
\begin{align}
& \beta(\cdot,x)\in C(\mathbb{R}_+,L^\infty(0,a_\dagger))\quad\text{with}
\\
& 0\leq \beta(a,x)\leq \beta_+\quad\text{a.e. in}\quad [0,a_\dagger]\times \mathbb{R}_+.
\end{align}
In addition we assume that $\beta(a,x)$ and $\mu(a,x)$ are Lipschitz continuous with respect to the second argument on bounded sets, uniformly on $a\in [0,a_\dagger]$. That is, for all $M>0$ there exists a constant $H(M)>0$ such that, if $x, \bar{x}\in [0,M]$, then
\begin{align}
|\mu(a,x)-\mu(a,\bar{x})| &\leq H(M)|x-\bar{x}|,\label{loclipmu} \\
|\beta(a,x)-\beta(a,\bar{x})| &\leq H(M)|x-\bar{x}|.\label{loclip}
\end{align}

\item\label{H'3} The weight functions are assumed to be non-negative and belong to $L^\infty(0,a_\dagger)$
\begin{equation*}
p,q\in L^\infty(0,a_\dagger)\text{,}\quad 0\leq p(a)\leq ||p||_\infty\quad\text{and}\quad 0\leq q(a)\leq q_+\quad \text{a.e. in}\quad [0,a_\dagger].
\end{equation*}

\item\label{fint} The initial distribution $f$ satisfies
\begin{equation*}
f\in L^1(0,a_\dagger),\; g(a)\geq 0\quad\text{ a.e. in }[0,a_\dagger].
\end{equation*}
\end{enumerate}

These assumptions can be found in \cite{Iannelli}. In order to study behavior of a population for large $t$, some additional properties of the birth rate $\beta$ and the weight function $p$ are needed. Namely, we suppose that there exist constants $a_2>b_2>b_1>a_1>0$ and $\delta>0$ such that
\begin{align}
\beta(a,x)&=0 \quad a\notin (a_1,a_2),
\\
\beta(a,x)&>\delta \quad \text{for }a\in(b_1,b_2),\label{beta}
\end{align}
and that there exist $p_2>p_1>0$ such that
\begin{equation}
p(a)>\delta\quad \text{for all }a\in [p_1,p_2].
\end{equation}



We begin our analysis by deriving an integral formulation to the model (\ref{gen1})--(\ref{genic1}). Our results are based on the reduction of the initial-boundary problem to the system of nonlinear integral equations for the number of newborns, denoted by
\begin{align}\label{rho0}
\rho(t)=n(0,t), \quad t>0,
\end{align}
and for the functions $P(t)$ and $Q(t)$.

As stated in Section 5.1 of \cite{Iannelli}, using the change of variables $a=x$ and $t=x+y$ and integrating along characteristic lines $y=C$, where $C$ is a constant, the balance equation (\ref{gen1}) becomes
\begin{equation}\label{n1}
   n(a,t) =\left\{\begin{array}{ll}
         \rho(t-a)e^{-\int_0^a\mu(v,P(v+t-a))dv}, & a<t,  \medskip \\
         f(a-t)e^{-\int_{a-t}^a\mu(v,P(v+t-a))dv}, & a\geq t.
        \end{array}\right.
\end{equation}
From (\ref{genbc1}), (\ref{rho0}) and (\ref{n1}) we obtain the system of integral equations:
\begin{equation}
\begin{aligned}
\rho(t) &=\int_0^t\beta(a,Q(t))\rho(t-a)e^{-\int_0^a\mu(v,P(v+t-a))dv}\,da \label{rho}\\
&\quad + \int_t^{\infty}\beta(a,Q(t))f(a-t)e^{-\int_{a-t}^a\mu(v,P(v+t-a))dv}\, da,
\end{aligned}
\end{equation}
\begin{equation}
\begin{aligned}
P(t) &= \int_0^t p(a)\rho(t-a)e^{-\int_0^a\mu(v,P(v+t-a))dv}\,da \\
&\quad + \int_t^{\infty}p(a)f(a-t)e^{-\int_{a-t}^a\mu(v,P(v+t-a))dv}\, da,\label{P0}
\end{aligned}
\end{equation}
and
\begin{equation}\label{Q0}
\begin{aligned}
Q(t) &= \int_0^t q(a)\rho(t-a)e^{-\int_0^a\mu(v,P(v+t-a))dv}\,da \\
&\quad + \int_t^{\infty}q(a)f(a-t)e^{-\int_{a-t}^a\mu(v,P(v+t-a))dv}\, da.
\end{aligned}
\end{equation}

Then we have   existence and uniqueness of a solution to the problem (\ref{rho})-(\ref{Q0}). More precisely, we have

\begin{theorem}\label{uniqeness}
Let assumptions \ref{H'1}--\ref{fint} hold. Then there exist unique non-negative functions ${\rho,P,Q\in C(\mathbb{R}_+)}$ satisfying problem (\ref{rho})-(\ref{Q0}).
\end{theorem}

This result is well-known and comes back to the analysis of age-structured population models in the papers Gurtin \& MacCamy \cite{Grt} and Chipot \cite{Cip}. The proof  (even for a more general model involving arbitrarily many weight functions) can be found, for example, in Section 5.1 of a recent monograph of Iannelli and Milner \cite{Iannelli}.

\section{Boundedness of solution}\label{boundedness}

The negative density-dependence is observed in biological populations as intraspecific competition or overcrowding effects, and investigated both practically and theoretically. Mathematical representation of the negative-density dependence begins with the Verhulst model for unstructured population, see for example \cite{Ian}, and the consequence of this type of regulation are bounded growth and stabilization of population around its carrying capacity. Effects of the negative density-dependence on the age-structured population are studied in \cite{we2}. Under the assumption that only members of the same age-class compete, the existence of a bounded solution has been proven. In what follows, we will prove the existence of a bounded solution considering more general mortality function which includes competition between different age classes. To this end we consider the problem (\ref{gen1})--(\ref{genic1}), where the non-negativity condition on $\mathcal{M}$ in \ref{H'1} is removed, and instead the following holds:
\begin{enumerate}[label=\textrm{($A_{\arabic*}$)},series=lafter]
\item\label{A2} There exist a function $\psi\in C(\mathbb{R}_+)$ such that \begin{equation*}
\mathcal{M}(a,x)\geq \psi(x)\geq -\sup \mu_0(a) \quad \text{for all }a\text{ and }x
\end{equation*}
where \begin{equation*}
\psi(\cdot) \text{ is non-decreasing,}\quad \lim\limits_{x\rightarrow \infty}\psi(x)=\infty.
\end{equation*}
\item\label{A1} There exists a constant $c>0$ such that $\beta(a)\le cp(a)$ for all $a$.
\end{enumerate}

Assumption \ref{A2} corresponds to the fact that for large populations mortality is increased by increase in population size and also generalizes mortality rate used in \cite{we2}. Note that for small populations this correlation does not need to hold. This allows us to include mortality functions that satisfy:
$\mu(a,P)$ is decreasing for $P\in(0,\delta)$ and increasing for $P>\delta$. These types of mortality functions can be related to the Allee effect to describe situations when, for small population sizes, increase in population size increases fitness by reducing mortality. Condition \ref{A2} implies that the density-dependent mortality rate is unbounded, which corresponds to our expectations since intraspecific competitions increases with population size.

Assumption \ref{A1} does not restrict birth rate $\beta$ or the weight function $p$, since $\beta$ is already bounded and $p$ is non-negative, according to \ref{H'2} and \ref{H'3}. However, it does provide a relation between individuals contribution to fecundity and mortality: individuals in every fertile age group are competing for resources and contributing to mortality rate of individuals of their age or older.

In what follows, we will show that the assumptions \ref{A2} and \ref{A1} are sufficient for boundedness of the functions $P(t)$, $Q(t)$ and $\rho(t)$ for all $t$. This improves the result in \cite{bound}, where the weight function $p(a)$ is supposed to be Lipschitz continuous. We begin by formulating the following lemma.


\begin{lemma}\label{lemma_boundedness}
Let $\rho$ be a non-negative continuous function on $[0,\infty)$ and let $\psi(x)$ satisfy  \ref{A2}. We define $\psi^{-1}(x)$ as \mbox{$\max\{y\ge0:\,\psi(y)=x\}$}. Let $\gamma=1-\psi(0)$. If there exist constants $c>0$ and $M> 1+\psi(\frac{\rho(0)}{c})$ such that
\begin{equation}\label{localbound}
\rho(t)\leq M\max\limits_{x\leq  t}\frac{\rho(x)}{\psi(\frac{\rho(x)}{c})+\gamma}\quad \text{for all } t,
\end{equation}
 then
 \begin{equation}\label{globalbound}
 \rho(t)\leq M\max\limits_{k\leq c\psi^{-1}(M-\gamma)}\frac{k}{\psi(\frac{k}{c})+\gamma}<\infty.
 \end{equation}
\end{lemma}
Proof of this lemma can be found in the Appendix~\ref{sec:appLem31}.

We now state and prove the main result of this section.

\begin{theorem}\label{theoremboundedness}
If the functions $\beta$, $\mu$, $f$, $p$ and $q$ satisfy \ref{H'2}--\ref{fint} with the additional assumptions $\ref{A2}$ and $\ref{A1}$, then the functions $\rho$, $P$ and $Q$ are bounded.
\end{theorem}

\begin{proof}
Using the variable changes $x=t-a$ and $v_{new}=v_{old}+x$ in the first integrals of (\ref{rho}) and (\ref{P0}), and assuming that $t\geq a_\dagger$, we obtain
\begin{align}
\label{pt}\rho(t) &=\int_0^t\beta(t-x,Q(t))\rho(x)e^{-\int_x^{t}\mu
(v-x,P(v))dv}\,dx, \\
\label{Pt} P(t)&=\int_{0}^t p(t-x)\rho(x)e^{-\int_x^{t}\mu(v-x,P(v))dv}\,dx,\\
\label{Qt} Q(t)&=\int_{0}^t q(t-x)\rho(x)e^{-\int_x^{t}\mu(v-x,P(v))dv}\,dx.
\end{align}
This together with assumption $\ref{A1}$ implies that
\begin{align}
P(t) &=\int_0^tp(t-x)\rho(x)e^{-\int_x^{t}\mu(v-x,P(v))dv}\,dx \nonumber  \\
& \ge \frac{1}{c}\int_0^t \beta(t-x,Q(t))\rho(x)e^{-\int_x^{t}\mu(v-x,P(v))dv}\,dx \nonumber \\
&\ge \frac{1}{c} \rho(t).\label{boundedbypopulation}
\end{align}
Using $\ref{A2}$, $\ref{A1}$ and  (\ref{boundedbypopulation}), from equation  (\ref{pt}) follows an estimate of $\rho$:
\begin{align*}
\rho(t)&\leq \int_{t-a_\dagger}^t\beta_{max}\rho(x)e^{-\int_x^{t}\psi(P(v))dv}\,dx\\
&\leq \int_{t-a_\dagger}^t\beta_{max}\rho(x)e^{-\int_x^{t}\psi(\frac{\rho(v)}{c})dv}\,dx.
\end{align*}
Multiplying  both the nominator and the denominator with $\psi(\frac{\rho(x)}{c})+\gamma>0$ we get.
\begin{align*}
p(t)&\leq \int_{t-a_\dagger}^t \beta_{max}\frac{\rho(x)}{\psi(\frac{\rho(x)}{c})+\gamma}(\psi(\frac{\rho(x)}{c})+\gamma)e^{-\int_x^{t}\psi(\frac{\rho(v)}{c})dv}\,dx\\
&\leq \beta_{max}\max\limits_{x<t}\frac{\rho(x)}{\psi\left(\frac{\rho(x)}{c}\right)+\gamma}\left(\int_{t-a_\dagger}^t\psi\left(\frac{\rho(x)}{c}\right)e^{-\int_x^{t}\psi\left(\frac{\rho(v)}{c}\right)dv}dx\right.\\
&\quad+\left.\int_{t-a_\dagger}^t e^{-\int_0^{t-x}\psi(\frac{\rho(v)}{c})dv}\,dx\right)
\\
&\leq \beta_{max}\max\limits_{x<t}\frac{\rho(x)}{\psi\left(\frac{\rho(x)}{c}\right)+\gamma}\left(\left[e^{-\int_x^t \psi(\frac{\rho(v)}{c})dv}\right]^{t}_{t-a_\dagger}+\int_{t-a_\dagger}^t \gamma\,dx\right)
\\
&\leq \beta_{max}(\gamma+a_\dagger)\max\limits_{x\leq t}\frac{\rho(x)}{\psi(\frac{\rho(x)}{c})+\gamma}.
\end{align*}
Lemma \ref{lemma_boundedness} infers that $\rho$ is bounded by
\begin{equation}\label{rho+}
M\max\limits_{k\leq c\psi^{-1}(M-\gamma)}\frac{k}{\psi(\frac{k}{c})+\gamma},
\end{equation}
where
\begin{equation*}
M=\max\left(\beta_{max}(\gamma+a_\dagger),\gamma+\psi(\frac{\rho(0)}{c})\right).
\end{equation*}
Finally, to prove that $P$ and $Q$ are bounded, it is sufficient to use boundedness of $\rho$ and $(\ref{P0})$
\end{proof}

\section{Local stability of the trivial equilibrium} \label{stability}
In order to investigate the local stability of the trivial equilibrium $(\rho,P,Q)=(0,0,0)$, problem (\ref{p})--(\ref{gen1}) is linearized. Let $(\rho,P,Q)=(z,\mathcal{P},\mathcal{Q})$ be a solution to (\ref{p})--(\ref{gen1}) and assume $(z(a,t),\mathcal{P}(t),\mathcal{Q}(t))$ is close to zero. In order to linearize
, we assume, in addition to the previous assumptions on $\beta$ and $\mu$, that $\beta(a,x)$ and $\mu(a,x)$ have continuous partial derivatives with respect to the second argument, uniformly in $a\in[0,a_\dagger]$. By linearization around zero we get
\begin{align}
\frac{\partial z(a,t)}{\partial t}+\frac{\partial z(a,t)}{\partial a}&=-\mu(a,0)z(a,t),\label{stablebalance0}
\\
z(0,t)&=\int_0^\infty \beta(a,0)z(a,t)\,da.\label{stablerenewal0}
\end{align}
If $z$ is known, $\mathcal{P}$ and $\mathcal{Q}$ can be calculated from formulas (\ref{p})-(\ref{q}).

In the age-structured population models, the net reproduction rate defined by
 \begin{equation}\label{R_00}
  R_0=\int_0^\infty\beta(a,0)e^{-\int_0^a \mu(\tau,0)d\tau}\,da
  \end{equation}
measures the number of offspring of an individual during its lifetime \cite{we1,we2}. It is often used as an indicator of the large time population behavior and a dichotomy between population survival for $R_0>1$ and extinction for $R_0\le 1$ has been proven in \cite{we1,we2}. Stability of the trivial equilibrium $(\rho,P,Q)=(0,0,0)$ of linear problem  (\ref{stablebalance0})--(\ref{stablerenewal0}) can be assessed using $R_0$ and we have the following result.

\begin{proposition}
  If $R_0<1$, then the solution of (\ref{stablebalance0})--(\ref{stablerenewal0}) converge to zero, and if $R_0> 1$, it increases to infinity. If $R_0=1$ then the solution is bounded and persistent. \end{proposition}
\begin{proof}
Let $\lambda$ be such that
\begin{equation}\label{lambda}
\int_0^\infty \beta(a,0)e^{\int_0^a \mu(v,0)dv-\lambda a}da =1
\end{equation}
Observe that the left-hand side of (\ref{lambda}) is a strictly decreasing continuous function with respect to $\lambda$, with values ranging from $\infty$ to $0$. Thus, $\lambda$ is well defined. By Theorem 3.2 and Theorem 3.3 in \cite{we1}, for $\sigma=\lambda$ and $z(0,t)\neq 0$, there exist constants $C_1,C_2>0$ such that
\begin{equation}\label{bounds}
C_1e^{\lambda t}\leq z(0,t) \leq C_2 e^{\lambda t}.
\end{equation}
If $R_0<1$, then $\lambda<0$ and if $R_0>1$, then $\lambda>0$. This, together with (\ref{bounds}), implies the theorem.
\end{proof}

\begin{remark}
As a consequence of the "Principal of Linearised Stability" in \cite{Desch}, it follows that asymptotic stability and instability of the linearised problem (\ref{stablebalance0})-(\ref{stablerenewal0}) implies asymptotic stability and instability respectively for the non-linear problem (\ref{p})-(\ref{genic1}). This means that for our original problem (\ref{p})-(\ref{genic1}) we can conclude that the trivial equilibrium is locally stable if $R_0<1$ and locally unstable if $R_0>1$
\end{remark}

We will not go into the details of \cite{Desch}, but for guidance we note that (\ref{gen1})--(\ref{genic1}) defines a family of operators $T(t):L^1(0,a_\dagger)\rightarrow C(\mathbb{R})$ which takes in an initial distribution $f(a)$ and gives the solution of (\ref{gen1})--(\ref{genic1}) evaluated at $t$, that is $n(\cdot,t)$. This family, as it turns out, is a semigroup and the Fréchet derivative of $T(t)$ is the corresponding operator derived from the linear problem (\ref{stablebalance0})--(\ref{stablerenewal0}).


In the next section we improve on our recent results about local stability by deriving conditions for persistence of the solution and for global extinction.
\section{Global stability analysis}\label{largetimebehavior}

The net reproduction rate $R_0$ defined by (\ref{R_00}) can be used to determine the large time behaviour of the solution to the problem (\ref{gen1})-(\ref{genic1}). Our next theorem claims that the functions $\rho$, $Q$ and $P$ are separated from zero if the net reproduction rate is greater than one, and that the functions $\rho$, $Q$ and $P$ converge to zero otherwise.

\begin{theorem}\label{persistencetheorem}
Under the assumptions that
\begin{align}\label{definiteness}
\psi(P)>0, \quad \text{for all }P>0,
\\
\beta(a,0)>\beta(a,Q),\quad\text{for all }Q>0,
\end{align}
the following holds:

\begin{enumerate}
\item[a)] If $R_0 \leq 1$, then $\rho(t)\rightarrow 0$, $P(t)\rightarrow 0$ and $Q(t)\rightarrow 0$ as $t\rightarrow\infty$.
\item[b)] If $R_0>1$, then there exists positive constants $0<a_k<b_k$, $k=1,2,3$, independent of $f$ such that
\begin{align*}
a_1\le \rho(t) \le b_1 ,\quad
a_2\le P(t) \le b_2 \quad\mbox{and}\quad a_3\leq Q(t)\leq b_3\quad\mbox{for large $t$.}
\end{align*}
\end{enumerate}
\end{theorem}
 To prove Theorem \ref{persistencetheorem}, we need the following lemma, which we formulate here and leave its proof for Appendix~\ref{sec:appLem52}.

\begin{lemma}\label{lemma1}
Let $\rho=\rho(t)$ be a non-negative function defined for $t>0$ and satisfying

\begin{align}\label{est:rho1}
c_1 \int_{t-b_2}^{t-b_1}\rho(\tau)\,d\tau \le \rho(t) \le c_2\int_{t-a_2}^{t-a_1}\rho(\tau)\,d\tau \quad\mbox{for} \quad t>a_2
\end{align}
where $0<a_1<b_1<b_2<a_2$ and $c_1$ and $c_2$ are positive constants.
Let also
\begin{align}\label{est:rho2}
\int_{t^*-\beta_2}^{t^*-\beta_1}\rho(\tau)\,d\tau \le c_3\Lambda \quad\mbox{for certain $t^*$}.
\end{align}
for some constants $\beta_1$ and $\beta_2$
Then for each $\hat t$ there exist constants $t_1$ and $c^*$ independent of $\Lambda$, $\rho$ and $t$ such that if $t^*\ge t_1$, then
\begin{align}\label{est:rho3}
\max_{t^*-\hat t\le \tau\le t^*}\rho(\tau) \le c^*\Lambda.
\end{align}
\end{lemma}
Now note that equation (\ref{rho}) together with the fact $\beta$ is bounded implies that the number of newborns satisfies the upper estimate in (\ref{est:rho1}). Since $P$ is bounded and $\beta$ is bounded from below on $(b_1,b_2)$ we have that the left-hand side of (\ref{est:rho1}) is true as well.
Lemma \ref{lemma1} now tells us that, for large $t$, if the integral over $\rho$ is small  i.e. $\Lambda$ is small, we have that $\rho$ also has to be small in the interval over which $\rho$ was integrated.

\begin{proof}[Proof of theorem \ref{persistencetheorem}]
a) Suppose that $R_0<1$,  $\varepsilon >0$ and $\rho^*=\limsup_{t\rightarrow\infty}\rho(t)$. From (\ref{pt}) it follows that
\begin{align*}
\rho(t) &\le \int_0^{\infty} \beta(a,Q(t))(\rho^*+\varepsilon) e^{-\int_0^a\mu(v,P(t))dv}\,da \\
&\le \int_0^{\infty} \beta(a,0)(\rho^*+\varepsilon) e^{-\int_0^a\mu(v,0)dv}\,da \\
&= (\rho^*+\varepsilon)R_0,
\end{align*}
for large $t$. Moreover, there exists a sequence $\{t_k\}$, $k=1,2,...$, such that $t_k\rightarrow\infty$ and $\rho(t_k) \ge \rho^*-\varepsilon$. From here we have that
\begin{align*}
\rho^*-\varepsilon \le (\rho^*+\varepsilon)R_0,
\end{align*}
and
$$ \rho^* \le \varepsilon \frac{1+R_0}{1-R_0},$$
implying that $\rho^*=0$.
This and equations (\ref{P0}) and (\ref{Q0}) lead us to the conclusion that $P(t)\rightarrow 0$ and $Q(t)\rightarrow 0$ as $t\rightarrow\infty$.
\\
Let us now consider the case when $R_0=1$. 
Using \ref{A2} and equation (\ref{pt}), we obtain
\begin{align}
\rho(t) &\le   \int_0^t \beta(a,0)e^{-\int_0^a\mu(v,P(v+t-a))dv}\rho(t-a)\,da \label{temprho0}\\
&\quad+ \int_t^{\infty}\beta(a,0)f(a-t)e^{-\int_{a-t}^a\mu(v,P(v+t-a))dv}\, da, \label{temprho}
\end{align}
and for $t>a_\dagger$ we have
\begin{align}\label{simple}
\rho(t) =   \int_0^t \beta(a,0)e^{-\int_0^a\mu(v,P(v+t-a))dv}\rho(t-a)\,da.
\end{align}
Similarly, 
from \ref{A2} and equation (\ref{Pt}), we get
\begin{align}
P(t) &\le  \int_0^tp(a)e^{-\int_0^a \mu_0(v)+\psi(P(v+t-a))dv}\rho(t-a)\,da \label{tempP0}\\
&\quad+ \int_t^{\infty}p(a)f(a-t)e^{-\int_{a-t}^a \mu_0(v)+\psi(P(v+t-a))dv}\, da. \label{tempP}
\end{align}
After the change of variables $x=t-a$, $y=v+t-a$ in (\ref{temprho0}) and (\ref{tempP0}), and $x=a-t$, $y=v+t-a$ in (\ref{temprho}) and (\ref{tempP}), we obtain
\begin{align*}
\rho(t) &\le\int_0^t\beta(t-x,0)\rho(x)e^{-\int_x^{t} \mu_0(y-x)+\psi(P(y))dy}\,dx
\\
&+ \int_0^{\infty}\beta(t+x,0)f(x)e^{-\int_0^{t}\mu_0(y+x)+ \psi(P(y))dy}\, dx,
\end{align*}
and
\begin{align*}
P(t) &\le\int_0^tp(t-x)\rho(x)e^{-\int_x^{t}\mu_0(y-x)+\psi(P(y))dy}\,dx
\\
&+ \int_0^{\infty}p(t+x)f(x)e^{-\int_0^{t}\mu_0(y+x)+\psi(P(y))dy}\, dx,
\end{align*}
which we can rewrite as
\begin{align*}
\rho(t) &\le\int_0^t\beta(t-x,0)e^{-\int_0^{t-x}\mu_0(y)dy}\rho(x)e^{-\int_x^{t} \psi(P(y))dy}\,dx
\\
&+ \int_0^{\infty}\beta(t+x,0)e^{-\int_0^{t+x}\mu_0(y)dy} f(x)e^{\int_0^x \mu_0(y)dy}e^{-\int_0^{t}\psi(P(y))dy}\, dx,
\end{align*}
and
\begin{align*}
P(t) &\le\int_0^tp(t-x)e^{-\int_0^{t-x}\mu_0(y)dy}\rho(x)e^{-\int_x^{t}\psi(P(y))dy}\,dx
\\
&+ \int_0^{\infty}p(t+x)e^{-\int_0^{t+x}\mu_0(y)dy}f(x)e^{\int_0^x \mu_0(y)dy}e^{-\int_0^{t}\psi(P(y))dy}\, dx.
\end{align*}
Multiplying both equations by $e^{\int_0^t\psi(P(y))dy}$ and introducing the notations
\begin{align}\label{ab}
\alpha_1(t)&=\rho(t)e^{\int_0^t\psi(P(y))dy}, \quad  \alpha_2(t)=P(t)e^{\int_0^t\psi(P(y))dy},
\\
M(a)&=\beta(a,0)e^{-\int_0^{a}\mu(v,0)dv}, \quad  S(a)=p(a)e^{-\int_0^a\mu_0(v)dv}
\end{align}
and
\begin{align}
F(a)&=f(a)e^{\int_0^a \mu_0(v)dv}
\end{align}
we get
\begin{align}
\alpha_1(t) &\le \int_0^t M(t-x)\alpha_1(x)\,dx + \int_0^{\infty}M(t+x)F(x)\, dx, \label{a1}\\
\alpha_2(t) &\le \int_0^t S(t-x)\alpha_1(x)\,dx + \int_0^{\infty}S(t+x)F(x)\, dx. \label{b1}
\end{align}
We note that $\alpha_1(t)$ is the number of newborns to a density-independent variant of the original problem (\ref{gen1})--(\ref{genic1}), with $R_0=1$ and initial age distribution $F(a)$.  Since $R_0=1$, then by Theorem $3.2$ in \cite{we1} with $\sigma=0$ we have that $\alpha_1(t)\le C$ and from (\ref{b1}) it follows that $\alpha_2(t)$ is also bounded. From (\ref{ab}) we get
$$\rho(t)e^{\int_0^t\psi(P(y))dy} \le C \quad\mbox{and}\quad P(t)e^{\int_0^t\psi(P(y))dy} \le C.$$
To prove convergence of $\rho$, $P$ and $Q$ we distinguish two cases.

If $\rho(t)\rightarrow 0$ then $P(t)\rightarrow 0$ and $Q(t)\rightarrow 0$ as $t\rightarrow\infty$, and the claim holds. If the above does not hold, then  $\int_0^{\infty}\psi(P(y))\,dy \le C$. In this case by assumption (\ref{definiteness}) there exists a sequence $t_k\rightarrow\infty$ as $k\rightarrow\infty$ such that $\varepsilon_k=P(t_k)\rightarrow 0$ as $k\rightarrow\infty$. We can in addition require that $|t_k-t_{k-1}|<1$. From equation (\ref{P0}) and condition \ref{H'3}, this implies that there exists a constant $c$ such that for all $k$,
$$\int_{t_k-p_2}^{t_k-p_1}\rho(\tau)\,d\tau\le c\varepsilon_k, \quad \mathrm{supp}(p)=[p_1,p_2].$$
By Lemma \ref{lemma1} with comments, for large enough $k$, we have $\max_{t_k-1\le \tau\le t_k}\rho(\tau) \le c^*\varepsilon_k$. By the requirement $|t_k-t_{k-1}|<1$ we can now conclude that $\rho(t)\rightarrow 0$. From (\ref{P0}) and (\ref{Q0}) we now also see that $P(t)\rightarrow 0$ and $Q(t)\rightarrow 0$.

Finally, we consider the case $R_0>1$ and we will show that $\rho(t)\ge \delta_1>0$ for large $t$. To this end, assume that there exist a sequence $t_k\rightarrow \infty$ as $k\rightarrow\infty$ such that $\varepsilon_k=\rho(t_k)\rightarrow 0$. Without loss of generality we can assume that $\rho(t_k)=\inf_{t_k-a_2\le t\le t_k}\rho(t)$. Since $\int_{t_k-b_2}^{t_k-b_1}\rho(\tau)\,d\tau \le c_1\rho(t_k)$, by Lemma \ref{lemma1} it follows that
$$\max_{t_k-\hat t\le \tau\le t_k}\rho(\tau) \le c_3\rho(t_k)=c_3\varepsilon_k,$$
which implies that
$$\max_{t_k-\hat t\le \tau\le t_k}P(\tau) \le c\varepsilon_k.$$
Using (\ref{pt}), for $t=t_k$ we get
$$\rho(t_k)\ge \rho(t_k)\int_0^{t_k}\beta(a,0)e^{-\int_0^a(\mu(v,P)-\mu(v,0))dv}\,da \ge \rho(t_k)R_0(1+o(\varepsilon_k)),$$
which is impossible because of $R_0>1$. We can now conclude that $a_1<\rho(t)<b_1$.

Suppose now that there exists a sequence $t_k\rightarrow\infty$ as $k\rightarrow\infty$, such that $P(t_k)\rightarrow0$ or $Q(t_k)\rightarrow 0$ as $k\rightarrow\infty$. Then from (\ref{Pt}) and (\ref{Qt}) it follows that
$$\varepsilon_k \ge c\int_{t_k-p'_2}^{t_k-p'_1}\rho(\tau)\,d\tau.$$
By Lemma \ref{lemma1} we get that
$$\liminf_{t\rightarrow\infty} \rho(t)=0,$$
which is impossible according to the previous part of the proof.
\end{proof}
Note that in the case when the maximum of $\beta(a,\cdot)$ and minimum of $\mu(a,\cdot)$ is not attained in 0, we can still come to similar conclusions using the same technique by redefining $R_0$. For example,
assume there exist functions $\mu_-$ and $\beta_+$ such that
\begin{align}
\mu(a,p)&\geq \mu_-(a)
\\
\beta(a,Q)&\leq \beta_+(a)
\end{align}
for all $a$.
Let
\begin{equation}
R_0^+=\int_0^\infty \beta(a,Q)e^{-\int_0^a\mu_-(v)dv}da
\end{equation}
then if $R_0^+<1$ we have that $\rho(t)\rightarrow0, P(t)\rightarrow 0, Q(t)\rightarrow 0$
%
\section{Permanence by positive density-dependence}\label{permanence}
Let us assume that influence of the Allee effect manifests though changes in the death rate. This means that in a small population, every increase in age-class decreases death rate. In other words, for every $a$, death rate $\mu(a,P)$ is a decreasing function of $P$ for $P\in(0,\delta)$.

We prove that if $R_0=1$ survival is possible due to the Allee effect.

\begin{theorem}
If $\mu(a,P)-\mu(a,0)<0$ for $P\in(0,\delta)$ and $R_0=1$, then $$\liminf_{t\rightarrow\infty}\rho(t)>0 \quad\mbox{and}\quad \liminf_{t\rightarrow\infty}P(t)>0.$$
\end{theorem}

\begin{proof}
Let
\begin{equation}\label{MF}
M(a)=\beta(a,Q(t))e^{-\int_0^{a}\mu(v,0)dv}
\end{equation}
Using (\ref{pt}), (\ref{MF}) and the assumption of the theorem, for $P<\delta$ and sufficiently large $t$ we obtain
\begin{align*}
\rho(t) &=  \int_0^tM(a)\rho(t-a)e^{-\int_0^a(\mu(v,P(v+t-a))-\mu(v,0))dv}\,da \\
&\ge \int_0^tM(a)\rho(t-a)\,da.
\end{align*}
To prove the claim, we suppose that  $\liminf_{t\rightarrow\infty}\rho(t)=0$. Then there exists a sequence $\{t_k\}$ such that $\rho(t_k)\rightarrow 0$ as $k\rightarrow\infty$. Without loss of generality we can assume that $\rho(t_k)=\inf_{t_k-a_2<t<t_k}\rho(t)$.
To show that $\rho(t)$ is small on $[t_{k-1},t_k]$, notice that
$$\varepsilon_k=\rho(t_k)=\int_0^t\beta(a,Q(t))\rho(t_k-a)e^{-\int_0^a\mu(v,P(v+t-a))dv}\,da,$$
and
$$\int_{t_k-a_2}^{t_k-a_1}\rho(\tau)\,d\tau \le c\varepsilon_k.$$
Then by Lemma \ref{lemma1} it follows that
$$\max_{t_k-t\le \tau\le t_k}\rho(\tau)\le c^*\varepsilon_k,$$
which implies that
$$P(t)=\int_0^tp(a)\rho(t-a)e^{-\int_0^a\mu(v,P(v+t-a))dv}\,da \le \delta$$
on $t_k-\hat t+M\le \tau \le t_k$. This also implies
$$\rho(t_k) > \int_0^{\infty}M(a)\rho(t_k)\,da=\rho(t_k),$$
which is impossible.
\end{proof}

We aim to prove conditions for extinction and permanence in the case of the Allee effect.
Under assumption that the weighted age-class function are constant $P(t)=\mathcal{P}$ and $Q(t)=\mathcal{Q}$ are constant over individuals life time, the weighted net reproduction rate $R(\mathcal{P},\mathcal{Q})$ is defined by
\begin{equation}
R(\mathcal{P},\mathcal{Q})=\int_0^\infty \beta(a,\mathcal{Q})e^{-\int_0^a \mu(v,\mathcal{P})dv}da.
\end{equation}
Notice that $R(0,0)=R_0$ as defined in (\ref{R_00}).
 \begin{lemma}\label{close enough}
 Assume that $R(\mathcal{P},\mathcal{Q})<1$ for $\mathcal{P}<P^*$, $\mathcal{Q}<Q^*$.
 Let $\rho^*>0$ be such that
 \begin{equation}
 \rho^*< \frac{P^*}{\int_0^\infty p(a)da}
 \quad \mbox{and}\quad
 \rho^*< \frac{Q^*}{\int_0^\infty q(a)da},
\end{equation}
  then if $\rho(t)<\rho*$ on some interval $t^*-a_\dagger<t<t^*$, then $\rho,Q,P\rightarrow 0$
\end{lemma}
\begin{remark}
Lemma \ref{close enough} tells us that, to conclude extinction, in some cases it might be enough to look at the number of newborns during a time period of the maximal lifetime.
\end{remark}
\begin{remark}
In \cite{Iannelli} it is concluded that $R_0<1$ implies asymptotic stability of the trivial equilibrium in the sense that there exists $\delta$ such that if $||f(a)||_1<\delta$ then $$\lim_{t\rightarrow \infty} ||n(\cdot,t)||_1=0.$$ Lemma \ref{close enough} is a similar result but differs in the way that we look at the number newborns to conclude pointwise convergence of $\rho,P$ and $Q$.
\end{remark}

As we will see the restriction on $\rho^*$ is chosen as to guarantee that if $\rho(t-v)<\rho^*$ on $0<v<a_\dagger$ then $P(t)<P^*$ and $Q(t)<Q^*$. Which then implies that, for all time, each  individual would live a life with net reproductive rate less then one. This in turn would imply extinction.
\begin{proof}
Let $\rho_+$ be the theoretical maximum of $\rho$ as in (\ref{rho+}) and let $||p||_\infty=\sup p$. Let
\begin{equation*}
\begin{aligned}
\varepsilon_P &=\frac{P^*-\int_0^\infty p(a)\rho^*da}{2||p||_\infty\rho_+}>0, \\
\varepsilon_Q &=\frac{Q^*-\int_0^\infty q(a)\rho^*da}{2||p||_\infty\rho_+}>0, \\
\delta_P &=(1-\frac{\rho^*\int_0^\infty p(a)da}{2P^*}), \\
\delta_Q &=(1-\frac{\rho^*\int_0^\infty q(a)da}{2Q^*}),
\end{aligned}
\end{equation*}
and furthermore let $\varepsilon=\min(\varepsilon_P,\varepsilon_Q)$. For $t^*<t<t^*+\varepsilon$ we have
\begin{equation}
\begin{aligned}
P(t) &= \int_0^t p(a)\rho(t-a)e^{-\int_0^a\mu(v,P(v+t-a))dv}\,da \\
&\quad + \int_t^{\infty}p(a)f(a-t)e^{-\int_{a-t}^a\mu(v,P(v+t-a))dv}\, da,
\\
&= \int_0^\varepsilon p(a)\rho(t-a)e^{-\int_0^a\mu(v,P(v+t-a))dv}\,da
\\
& \quad+\int_\varepsilon^t p(a)\rho(t-a)e^{-\int_0^a\mu(v,P(v+t-a))dv}\,da
\\
&<
\varepsilon_P ||p||_\infty \rho_+ +\int_0^\infty p(a)\rho^*da=\delta_P P^*
\end{aligned}
\end{equation}
In the same way for $t^*<t<t^*+\varepsilon$ we have
\begin{equation}
Q(t)<\delta_Q Q^*.
\end{equation}
Let
\begin{equation}
R_1=\max_{\mathcal{P}\leq \delta_P P^*, \mathcal{Q}\leq \delta_Q Q^*}R(\mathcal{Q},\mathcal{P}).
\end{equation}
Now we estimate the number of newborns $\rho$ on the interval $t^*\leq t<t^*+\varepsilon$,
\begin{align}
\rho(t) &=\int_0^t\beta(a,Q(t))\rho(t-a)e^{-\int_0^a\mu(v,P(v+t-a))dv}\,da + \int_t^{\infty}\beta(a,Q(t))f(a-t)e^{-\int_{a-t}^a\mu(v,P(v+t-a))dv}\, da\nonumber
\\
&\leq \rho^*\int_\varepsilon^t\beta(a,Q(t)) e^{-\int_0^a\mu(v,P(v+t-a))dv}\,da +\rho_+\int_0^\varepsilon \beta(a,Q(t))e^{-\int_0^a\mu(v,P(v+t-a))dv}\,da \label{replace}
\\
& \quad +\rho^*\int_t^{\infty}\beta(a,Q(t))e^{-\int_{a-t}^a\mu(v,P(v+t-a))dv}\, da\nonumber
\\
&\leq \rho^*R_1+\varepsilon \rho_+.
\end{align}
If $\varepsilon\leq \frac{\rho^*(1-R_1)}{2\rho_+}$, from the above inequality it follows that $\rho\leq \frac{1+R_1}{2}\rho^*<\rho^*$ on $t^*<t<t^*+\varepsilon $.
So let
\begin{equation}
\gamma=\min(\varepsilon,\frac{\rho^*(1-R_1)}{2\rho_+})
\end{equation}
then $\rho(t)<\frac{1+R_1}{2}\rho^*$ on $t^*<t<t^*+\gamma$. Iterating a finite amount of time we get that  $\rho(t)<\frac{1+R_1}{2}\rho^*$ on $t^*<t<t^*+a_\dagger$. We can use this result yet again to conclude that $\rho(t)<(\frac{1+R_1}{2})^2\rho^*$ on $t^*+a_\dagger<t<t^*+2a_\dagger$, $\rho(t)<(\frac{1+R_1}{2})^3\rho^*$ on $t^*+a_\dagger<t<t^*+3a_\dagger$ and so on. This implies that $\rho$ converges to zero. From (\ref{P0}) and (\ref{Q0}) we see that also $Q$ and $P$ converge to zero.
\end{proof}

In the next theorem we will see that a population that is not converging to zero is necessarily persistent.
\begin{theorem}\label{twooptions}
If $R_0<1$, then either $\rho,Q,P\rightarrow 0$ or there exist $\varepsilon_\rho,\varepsilon_P,\varepsilon_Q>0$ such that $\rho>\varepsilon_\rho, P>\varepsilon_P, Q>\varepsilon_Q$ i.e the population is persistent.
\end{theorem}
\begin{proof}
If $\varepsilon_\rho$ exist, then from \ref{H'2} and \ref{H'3} it follows that $\varepsilon_P$ and $\varepsilon_Q$ necessarily exists.
Conversely, if $\varepsilon_\rho$ does not exist, then there exist a sequence $t_k>a_\dagger$ such that $t_k\rightarrow \infty $ and $\rho_{t_k}\rightarrow 0$ as $k\rightarrow\infty$. Let $\varepsilon>0$ be arbitrary. There exist $K$ such that if $k>K$ then $\rho(t_k)<\varepsilon$. For these $k$ we have
\begin{equation*}
\int_0^{t_k}\beta(a,Q(t_k))\rho(t_k-a)e^{-\int_0^a\mu(v,P(v+t_k-a))dv}\,da<\varepsilon.
\end{equation*}
This together with Theorem \ref{theoremboundedness} and \ref{H'2} implies that there exists $C>0$ such that
 \begin{equation*}
 \int_{t-b_2}^{t-b_1}\rho(t-a)da<C\varepsilon
 \end{equation*}
 Using Lemma \ref{lemma1} we get that there exist constants $t_1$ and $c^*$ independent from $\varepsilon$
such that if $t_k>t_1$ then
\begin{equation*}
\max_{t_k-a_\dagger\leq \tau \leq t_k}\rho(\tau)\leq c^*\varepsilon.
\end{equation*}
Choosing $\varepsilon$ to satisfy $c^*\varepsilon<\min(\frac{P^*}{\int_0^\infty p(a)da}, \frac{Q^*}{\int_0^\infty q(a)da})$ Lemma \ref{close enough} now guarantees that $\rho\rightarrow 0$ and we reach a contradiction.
\end{proof}

\section{Equilibrium points of \eqref{rho}--\eqref{Q0} and stability analysis}
Here we derive the equilibrium points to our model,  i.e. solutions constant in time, and then continue with performing a stability analysis for these equilibrium points. A similar analysis of equilibria and their stability has been done in  to Section 5.3 and Chapter 6 of \cite{Iannelli} for a more general model including arbitrary many weighted sizes. We derive some explicit criteria in our case by means of the characteristic equation \eqref{main} below and then relate the obtained threshold parameter to an analogue of the net reproductive rate $R_0$.

We assume that $(\rho^*,P^*,Q^*)$ are constant solutions. From equations $(\ref{rho}),(\ref{P0})$ and $(\ref{Q0})$ we get
\begin{equation*}
\begin{aligned}
\rho^* &=\int_0^\infty\beta(a,Q^*)\rho^*e^{-\int_0^a \mu(v,P^*)\,dv}\,da, \\
P^* &=\int_0^\infty p(a)\rho^*e^{-\int_0^{a}\mu(v,P^*)\,dv}\,da,\\
Q^* &=\int_0^\infty p(a)\rho^*e^{-\int_0^{a}\mu(v,P^*)\,dv}\,da.
\end{aligned}
\end{equation*}
If $\rho^*=0$ it follows that $(\rho^*,P^*,q^*)=(0,0,0)$. Otherwise, we have that
\begin{equation*}
\begin{aligned}
\rho^* &=\frac{P^*}{\int_0^\infty p(a)e^{-\int_0^{a}\mu(v,P^*))dv}da},
\end{aligned}
\end{equation*}
\begin{equation*}
Q^*=P^*\Gamma(P^*),
\end{equation*}
where
\begin{equation*}
   \Gamma(P^*):= \frac{\int_0^\infty q(a)e^{-\int_0^a \mu(v,P^*)dv}da}{\int_0^\infty p(a)e^{-\int_0^a \mu(v,P^*)dv}da}
\end{equation*}
and
\begin{equation}\label{solveforP^*}
1 =\int_0^\infty\beta(a,P^*\Gamma(P^*))e^{-\int_0^a \mu(v,P^*)dv}da.
\end{equation}
So first we solve the last equation (\ref{solveforP^*}) with respect to $P^*$ and then compute $Q^*$ and $\rho^*$ from the other two equations. If $R_0>1$ then a solution necessarily exists since the right-hand side of (\ref{solveforP^*}) goes continuously from $R_0$ to zero as $P^*$ goes to infinity.

Let $\rho^*(a)$ be an equilibrium point. We set $$n(a,t)=\rho^*(a)+z(a,t), \quad \text{where $z$ is a small perturbation} $$
For this analysis we need to assume that $\mu$ and $\beta$ are differentiable with respect to the second argument. We denote $\mu_P$ and $\beta_Q$ the derivative of $\mu$ and $\beta$ with respect to the second argument.
Linearising equation (\ref{gen1}) we get
\begin{align*}
&\frac{\partial z(a,t)}{\partial t}+\frac{\partial \rho^*(a))}{\partial a}+\frac{\partial z(a,t)}{\partial a}
\\
&=-\mu(a,P(t))(\rho^*(a)+z(a,t))\\&=
-(\mu(a,P^*)+\mathcal{P}(t)\mu_P(a,P^*))(\rho^*(a)+z(a,t))
\\
&= -\mu(a,P^*)\rho^*(a)-\mathcal{P}(t           )\mu_P(a,P^*)\rho^*(a)-\mu(a,P^*)z(a,t).
\end{align*}
From equation (\ref{genbc1}) it follows
\begin{align*}
\rho^*(0)+z(0,t) &=\int_0^{a_\dagger}\beta(a,Q(t))\left(\rho^*(a)+z(a,t)\right)da
\\
&= \int_0^{\infty}(\beta(a,Q^*)+\mathcal{Q}(t)\beta_Q(a,Q^*))\left(\rho^*(a)+z(a,t)\right)da,
\\
&=\int_0^\infty \beta(a,Q^*)\rho^*(a)+\beta(a,Q^*)z(a,t)+\mathcal{Q}(t)\beta_Q(a,Q^*)\rho^*(a)\,da,
\end{align*}
and from \ref{p} and \ref{q} we get
\begin{align*}
P(t) &=\int_0^\infty p(a)(\rho^*(a)+z(a,t))da=P^*+\mathcal{P}(t),\quad \mathcal{P}(t)=\int_0^\infty p(a)z(a,t)da,
\\
Q(t) &=\int_0^\infty q(a)(\rho^*(a)+z(a,t))da=P^*+\mathcal{Q}(t),\quad \mathcal{Q}(t)=\int_0^\infty q(a)z(a,t)da.
\end{align*}
Using the fact that $\rho^*(a)$ is an equilibrium point we get
\begin{align*}
\frac{\partial z(a,t)}{\partial t}+\frac{\partial z(a,t)}{\partial a}
&= -\mathcal{P}(t)\mu_P(a,P^*)\rho^*(a)-\mu(a,P^*)z(a,t),
\\
z(0,t)&=\int_0^\infty \left(\mathcal{Q}(t)\beta_Q(a,Q^*)\rho^*(a)+\beta(a,Q^*)z(a,t)\right)\,da,
\\
 \mathcal{P}(t)&=\int_0^\infty p(a)z(a,t)\,da,
\\
 \mathcal{Q}(t)&=\int_0^\infty q(a)z(a,t)\,da.
\end{align*}
We look for solutions of the following form
\begin{align*}
&z(a,t)=g(a)e^{\lambda t}, \quad z(0,t)=C_1e^{\lambda t},
\quad \mathcal{P}(t)=C_2e^{\lambda t},
\quad \mathcal{Q}(t)=C_3e^{\lambda t}.
\end{align*}
Substituting these data we get the following
  \begin{align}
g(a)\lambda +\frac{d g(a)}{d a}
&= -C_2\mu_P(a,P^*)\rho^*(a)-\mu(a,P^*)g(a),\label{firstorder}
\\
C_1 &=\int_0^\infty C_3\beta_Q(a,Q^*)\rho^*(a)+\beta(a,Q^*)g(a)\,da, \label{eqC1}
\\
C_2 &=\int_0^\infty p(a)g(a)\,da, \label{C_3fromC_1}
\\
C_3 &=\int_0^\infty q(a)g(a)\,da. \label{eqC3}
 \end{align}
Equation (\ref{firstorder}) can solved with respect to $g(a)$ using the integrating factor method, with the solution
\begin{align}
&g(a)= \frac{C_1-\int_0^a C_2\mu_P(\sigma,P^*)\rho^*(\sigma)e^{ \sigma\lambda +\int_0^\sigma\mu(\tau,P^*)d\tau} d\sigma}{e^{ a\lambda +\int_0^a\mu(\sigma,P^*)d\sigma}}.
\label{C_1fromC_3}
 \end{align}
Inserting (\ref{C_1fromC_3}) in the system of equations (\ref{eqC1})--(\ref{eqC3}), we get
  \begin{align*}
C_1 &=\int_0^\infty\left( C_3\beta_Q(a,Q^*)\rho^*(a)\right.
\\
& \quad \left. +\beta(a,Q^*)\frac{C_1-\int_0^a C_2\mu_P(\sigma,P^*)\rho^*(\sigma)e^{ \sigma\lambda +\int_0^\sigma\mu(\tau,P^*)d\tau} d\sigma}{e^{\int_0^a \lambda +\mu(\sigma,P^*)d\sigma}}\right)da,
\\
C_2 &= \int_0^\infty p(a)\frac{C_1-\int_0^a C_2\mu_P(\sigma,P^*)\rho^*(\sigma)e^{ \sigma\lambda +\int_0^\sigma\mu(\tau,P^*)d\tau} d\sigma}{e^{ a\lambda +\int_0^a\mu(\sigma,P^*)d\sigma}}\,da,
\\
C_3 &=\int_0^\infty q(a)\frac{C_1-\int_0^a C_2\mu_P(\sigma,P^*)\rho^*(\sigma)e^{ \sigma\lambda +\int_0^\sigma\mu(\tau,P^*)d\tau} d\sigma}{e^{ a\lambda +\int_0^a\mu(\sigma,P^*)d\sigma}}\,da.
 \end{align*}
 To simplify calculations we introduce the following notations
 \begin{align*}
 A_1 &=\int_0^\infty \beta_Q(a,Q^*)\rho^*(a)\,da, \\
 A_2(\lambda) &=-\int_0^{\infty}\beta(a,Q^*)\int_0^a \mu_P(\sigma,P^*)\rho^*(\sigma)e^{-\sigma\lambda-\int_{a-\sigma}^a\mu(\tau,P^*)d\tau}d\sigma da, \\
A_3(\lambda) &=\int_0^\infty \beta(a,Q^*)e^{-a\lambda-\int_0^a \mu(\tau,P^*)d\tau}da, \\
 A_4(\lambda) &=-\int_0^\infty p(a)\int_0^a\mu_P(\sigma,P^*)\rho^*(\sigma)e^{ -\sigma\lambda-\int_{a-\sigma}^a\mu(\tau,P^*)d\tau}d\sigma da, \\
 A_5(\lambda) &=\int_0^\infty p(a)e^{ -a\lambda-\int_0^a\mu(\tau,P^*)d\tau}da, \\
 A_6(\lambda) &=-\int_0^\infty q(a)\int_0^a\mu_P(\sigma,P^*)\rho^*(\sigma)e^{ -\sigma\lambda-\int_{a-\sigma}^a\mu(\tau,P^*)d\tau}d\sigma da, \\
A_7(\lambda) &= \int_0^\infty q(a)e^{ -a\lambda-\int_0^a\mu(\tau,P^*)d\tau}da.
 \end{align*}
 With the notations above the system of equations can be written
 \begin{align}\label{matrixeq}
 \left(\begin{matrix}
 A_3(\lambda)-1 && A_2(\lambda) && A_1
 \\
 A_5(\lambda) && A_4(\lambda)-1 && 0
 \\
 A_7(\lambda) && A_6(\lambda) && -1
\end{matrix}\right)
\left(
\begin{matrix}
C_1 \\ C_2 \\ C_3
\end{matrix}
\right)
=0.
 \end{align}
 There exist small non-zero solutions $C_1,C_2$ and $C_3$ to (\ref{matrixeq}) if and only if
 \begin{equation}\label{main}
 \text{det}\left(\begin{matrix}
 A_3(\lambda)-1 && A_2(\lambda) && A_1
 \\
 A_5(\lambda) && A_4(\lambda)-1 && 0
 \\
 A_7(\lambda) && A_6(\lambda) && -1
\end{matrix}\right)
=0.
 \end{equation}
For the trivial equilibrium we get
 \begin{equation}
A_3(\lambda)=
 \int_0^\infty \beta(a,0)e^{-a\lambda-\int_{0}^{a}\mu(\tau,0)d\tau}da=1. \label{realgood}
 \end{equation}
 If we let $\text{Re}(\lambda)=\gamma$ and Im$(\lambda)=\phi$ (\ref{realgood}) turns into
 \begin{align}
 \int_0^\infty\beta(a,0)e^{-a\gamma-\int_0^{a}\mu(\tau,0)d\tau}e^{-a\phi i}\,da
 &=\text{Re}\,A_3(\gamma,\phi)+i\,\text{Im}\,A_3(\gamma,\phi)=1\label{imaginary}
 \end{align}
 where
 \begin{align*}
 \text{Re}\,A_3(\gamma,\phi) &= \int_0^\infty\beta(a,0)e^{-a\gamma-\int_0^a\mu(\tau,0)d\tau}\cos(a\phi)\,da, \\
 \text{Im}\,A_3(\gamma,\phi)&=
 -\int_0^\infty\beta(a,0)e^{-a\gamma-\int_0^a \mu(\tau,0)d\tau}\sin(a\phi)\,da.
 \end{align*}

 We observe that $\text{Re}\,A_3(\cdot,0):\mathbb{R}\rightarrow (0,\infty)$ is strictly decreasing and onto, so that under the made assumptions, the equation
\begin{equation*}
 \text{Re}\,A_3(\gamma,0)=1
 \end{equation*}
  has a unique solution $\gamma^*$. Furthermore, $\text{Re}\,A_3(\gamma^*,\cdot)$ has its unique maximum when $\phi=0$. Then for all solutions with $\phi\neq 0$ to equation (\ref{imaginary}), we have $\gamma<\gamma^*$.
  Let
  \begin{equation*}
  R_0=\text{Re}\,A_3(0,0)=\int_0^\infty\beta(a,0)e^{-\int_0^a \mu(\tau,0)d\tau}da
  \end{equation*}
  If $R_0<1$, we have that $\gamma^*<0$, implying $\gamma<0$ for all solutions and we can conclude that the trivial equilibrium point is stable.
  If $R_0>1$, we have that $\gamma^*>0$ and the trivial equilibrium point is unstable.


\begin{appendix}

\section{Proof of Lemma \ref{lemma_boundedness}}\label{sec:appLem31}

If the right-hand side of $(\ref{globalbound})$ is greater or equal to the right-hand side of $(\ref{localbound})$, then $(\ref{globalbound})$ follows from $(\ref{localbound})$. Let $\gamma=1-\psi(0)$. Assume that there exist $T\geq 0$ such that
\begin{equation*}\max\limits_{x\leq T}\frac{\rho(x)}{\psi(\frac{\rho(x)}{c})+\gamma}>\max\limits_{0<k\leq c\psi^{-1}(M-\gamma)}\frac{k}{\psi(\frac{k}{c})+\gamma}.
\end{equation*}
Due to the condition on $M$, we have that $T>0$. There exists $0<t_1\leq T$ such that
\begin{equation*}\frac{\rho(t_1)}{\psi(\frac{\rho(t_1)}{c})+\gamma}=\max\limits_{x\leq T}\frac{\rho(x)}{\psi(\frac{\rho(x)}{c})+\gamma}=\max\limits_{x\leq t_1}\frac{\rho(x)}{\psi(\frac{\rho(x)}{c})+\gamma},
\end{equation*}
and since
\begin{equation*}\frac{\rho(t_1)}{\psi(\frac{\rho(t_1)}{c})+\gamma}>\max\limits_{0<k\leq c\psi^{-1}(M-\gamma)}\frac{k}{\psi(\frac{k}{c})+\gamma},
\end{equation*}
 we have that $\rho(t_1)>c\psi^{-1}(M-\gamma)$. Note that by the definition of $\psi^{-1}$ this means that $\psi(\frac{\rho(t_1)}{c})>\psi(\frac{c\psi^{-1}(M-\gamma)}{c})=M-\gamma$. Now from $(\ref{localbound})$  we get
\begin{equation*}
\rho(t_1)\leq M\max\limits_{x\leq t_1}\frac{\rho(x)}{\psi(\frac{\rho(x)}{c})+\gamma}=M\frac{\rho(t_1)}{\psi(\frac{\rho(t_1)}{c})+\gamma}<\rho(t_1)
\end{equation*}
and we reach a contradiction. This proves the lemma.

\section{Proof of lemma \ref{lemma1}}\label{sec:appLem52}

\begin{figure}[!ht]
\begin{tikzpicture}[scale=1.1]
\draw[very thin,color=gray] (0,0) grid (9,6);
\draw[->,name path=xaxis] (0,0) -- (8.2,0) node[right] {$s$};
\draw[->,name path=yaxis] (0,0) -- (0,5.2) node[above] {$\tau$};

\draw[very thick,color=blue,name path=plot,domain=1:3] plot (\x,{-1+2*\x});
\draw[very thick,color=blue,name path=plot,domain=6:8] plot (\x,{-11+2*\x});
\draw[very thick,color=blue,name path=plot,domain=1:6] plot (\x,1);
\draw[very thick,color=blue,name path=plot,domain=3:8] plot (\x,5);
\draw[very thick,color=pink,name path=plot,domain=1:5] plot (3,\x);
\draw[very thick,color=pink,name path=plot,domain=1:5] plot (6,\x);
\draw[->,very thick,color=black,name path=plot,domain=1.5:2.33] plot (\x,6-\x);
\draw[<-,very thick,color=black,name path=plot,domain=6.5:7.5]
 plot (\x,8.5-\x);
 \draw[->,very thick,color=black,name path=plot,domain=0.75:1.75]
 plot(\x,3.75-\x);
 \draw[-,very thick,color=black,name path=plot,domain=0.75:1.75]
 plot(\x,3.75-\x);

\node[scale=1] (X) at (3,0) {$|$};
\node[below,scale=1] (X) at (3,0) {\rotatebox{90}{$t-\beta_1-b_2$}};
\node[scale=1] (X) at (8,0) { $|$};
\node[below,scale=1] (X) at (8,0) {\rotatebox{90}{$t-\beta_1-b_1$}};
\node[scale=1] (X) at (1,0) { $|$};
\node[below,scale=1] (X) at (1,0) {\rotatebox{90}{$t-\beta_2-b_2$}};
\node[scale=1] (X) at (6,0) { $|$};
\node[below, scale=1] (X) at (6,0) { \rotatebox{90}{$t-\beta_2-b_1$}};
\node[scale=1] (X) at (2,0) { $|$};
\node[below,scale=1] (X) at (2,0) { \rotatebox{90}{$t-\beta_2-b'_2$}};
\node[scale=1] (X) at (7,0) { $|$};
\node[below,scale=1] (X) at (7,0) { \rotatebox{90}{$t-\beta_1-b'_1$}};
\node[left,scale=1] (X) at (0,1) {$t-\beta_2$};
\node[scale=1] (X) at (0,1) { $-$};
\node[left,scale=1] (X) at (0,5) {$t-\beta_1$};
\node[scale=1] (X) at (0,5) { $-$};
\node[above] (X) at (1.5,4.5) {$\tau=s+b_2$};
\node[below] (X) at (7.5,1) {$\tau=s+b_1$};
\fill[teal!] (2,1) -- (2,3) -- (3,5) -- (7,5) -- (7,3)--(6,1)--(6,1);
\end{tikzpicture}
\caption{\label{figure}
The parallelogram is the area over which $\rho$ is integrated in (\ref{parallelogram}). The colored area is the domain over which $\rho$ is integrated in (\ref{teal}).
}\label{fig1}
\end{figure}
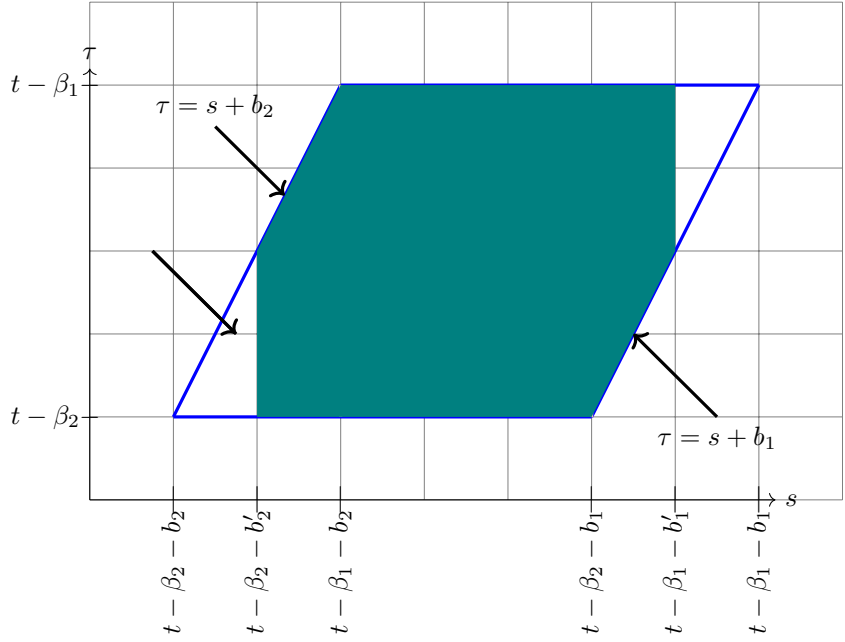

We will repeatedly use the following identity:
 For constants ${\beta_1<\beta_2}$ and $b_1<b_2$ we have
\begin{align}\label{order}
\int_{t-\beta_2}^{t-\beta_1}\int_{\tau-b_2}^{\tau-b_1}\rho(s)\,ds\,d\tau= \int_{t-\beta_2-b_2}^{t-\beta_1-b_1}\int_{\max(t-\beta_2,s+b_1)}^{\min(t-\beta_1,s+b_2)}\rho(\tau)\,d\tau\,ds.
\end{align}
Figure \ref{figure} in shows this fact.

Using the left hand side of (\ref{est:rho1}) to estimate (\ref{est:rho2}) from below we get
\begin{align}
&c_3 \Lambda \geq \int_{t^*-\beta_2}^{t^*-\beta_1}\rho(\tau)d\tau \geq \int_{t^*-\beta_2}^{t^*-\beta_1}\int_{\tau-b_1}^{\tau-b_2}\rho(\tau)\,d\tau\,ds.
\end{align}
Let $b'_1,b'_2$ be constants such that $b_1<b'_1<b'_2<b_2$. By (\ref{order}) we get
\begin{align}\label{parallelogram}
c_3\Lambda &\geq c_1\int_{t^*-\beta_2-b_2}^{t-\beta_1-b_1}\int_{\max(t^*-\beta_2,s+b_1)}^{\min(t^*-\beta_1,s+b_2)}\rho(\tau)\,d\tau\,ds
\\
\label{teal}&\geq c_1\int_{t-\beta_2-b'_2}^{t-\beta_1-b'_1}\int_{\max(t-\beta_2,s+b_1)}^{\min(t-\beta_1,s+b_2)}\rho(\tau)\,d\tau\,ds.
\end{align}
Since $\int_{\max(t-\beta_2,s+b_1)}^{\min(t-\beta_1,s+b_2)}\rho(\tau)d\tau>\gamma>0$ is bounded from below on the interval
\newline
${s\in[t-\beta_2-b'_2,t-\beta_1-b'_1]}$ (see Figure~\ref{fig1}) we get that
\begin{equation}
\gamma c_1\int_{t^*-\beta_2-b'_2}^{t^*-\beta_1-b'_1}\rho(s)\,ds \le c'_k\Lambda.
\end{equation}

Iterating $k$ times we get
\begin{align}\label{est:rho4}
\int_{t^*-\beta_2-kb'_2}^{t^*-\beta_1-kb'_1}\rho(s)\,ds \le c'\Lambda,
\end{align}
for some $c'>0$ depending on $b_1,b_2,b'_1,b'_2,m,\beta_1,\beta_2$ and $k$.
Using the right-hand side of  (\ref{est:rho1}) we derive from the previous inequality

\begin{align*}
\rho(T)\le c_2\int_{T-a_2}^{T-a_1}\rho(\tau)\,d\tau \le c_2\int_{t^*-\beta_2-kb'_2}^{t^*-\beta_1-kb'_1}\rho(\tau)\,d\tau \le c_2c'\Lambda,
\end{align*}
where $d_1=t^*-\beta_2-kb'_2+a_2 \le T \le t^*-\beta_1-kb'_1+a_1=d_2$. Here $d_1$ and $d_2$ are  required to be bigger then $a_2$ which gives us the value of $t_1$
\begin{equation}
t_1=\beta_2+kb'_2.
\end{equation}
We assume that $k$ is chosen large enough to satisfy  $d_2-d_1 >a_2-a_1$. Then for $T\in[d_2,d_2+a_1]$ we have
$$\rho(T) \le c\int_{T-a_2}^{T-a_1}\rho(\tau)\,d\tau\le c\int_{T-a_2}^{T-a_1}c'\Lambda\,d\tau \le cc'(a_2-a_1)\Lambda$$
according to (\ref{est:rho4}). Hence, there exist a constant $c_1$ such that
$$\rho(T)\le c\Lambda \quad\mbox{for $d_1\le T\le d_2+a_2$.}$$
Continuing this procedure we get a constant $\widetilde c$ such that
$$\rho(T) \le \widetilde c\Lambda \quad\mbox{for $d_1\le T\le d_2+la_2$.}$$
Choosing $k$ large enough to satisfy $d_1\le t^*-\hat t$ and then choosing $l$ large enough to satisfy $d_2+la_2\ge t^*$, we arrive at (\ref{est:rho3}).


\end{appendix}

\nocite{*}
\bibliographystyle{plain}


%

\end{document}